\RequirePackage{fix-cm}
\documentclass[smallextended]{svjour3}       
\smartqed  
\usepackage{graphicx}

\usepackage{amsmath}
\usepackage{amsfonts}
\usepackage{amssymb}
\usepackage{amscd}

\usepackage{hyperref}
\begin{document}

\title{ Generalizations of ${r}$-ideals of commutative rings
}


\author{Emel Aslankarayigit Ugurlu        
}


\institute{E. Aslankarayigit Ugurlu \at
             Department of Mathematics, Marmara University, Istanbul, Turkey. \\ \email{emel.aslankarayigit@marmara.edu.tr}           
     }

\date{Received: 29 February 2020 / Accepted: date}

\maketitle

\begin{abstract}
In this study, we present the generalization of the concept of $r$-ideals in
commutative rings with nonzero identity. Let $R\ $be a commutative ring with
$0\neq1$ and $L(R)$ be the lattice of all ideals of $R.\ $Suppose that
$\phi:L(R)\rightarrow L(R)\cup\left\{  \emptyset\right\}  $ is a function. A
proper ideal $I\ $of $R\ $is called a $\phi-r$-ideal of $R\ $if whenever
$ab\in I$ and $Ann(a)=(0)$ imply that $b\in I$ for each $a,b\in R.$ In
addition to giving many properties of $\phi-r$-ideal, we also examine the
concept of $\phi-r$-ideal in trivial ring extension and use them to
characterize total quotient rings.

\end{abstract}

\section{Introduction}
\label{intro}
In this article, we focus only on commutative rings with nonzero identity and
nonzero unital modules. Let $R$ always denote such a ring and $M\ $denote such
an $R$-module. $L(R)$ denotes the lattice of all ideals of $R.$ Also the
radical of $I$ is defined as $\sqrt{I}:=\{r\in R$ $|$ $r^{k}\in I$ for some
$k\in%
\mathbb{N}
\}$.

\bigskip

A proper ideal $I$ of a commutative ring $R$ is \textit{prime }if whenever
$a_{1},a_{2}\in R$ with $a_{1}a_{2}\in I$, then $a_{1}\in I$ or $a_{2}\in I,$
\cite{atiyah}. In 2003, the authors said that if whenever $a_{1},a_{2}\in R$
with $0_{R}\neq a_{1}a_{2}\in I$, then $a_{1}\in I$ or $a_{2}\in I,$ a proper
ideal $I$ of a commutative ring $R$ is \textit{weakly prime }\cite{AS}. In
\cite{BS}, Bhatwadekar and Sharma defined a proper ideal $I$ of an integral
domain $R$ as \textit{almost prime (resp. $n$-almost prime)} if for
$a_{1},a_{2}\in R$ with $a_{1}a_{2}\in I-I^{2}$, (resp. $a_{1}a_{2}\in
I-I^{n},n\geq3$) then $a_{1}\in I$ or $a_{2}\in I$. Later, Anderson and
Batanieh, in \cite{AB}, introduced a concept which covers all the previous
definitions in a commutative ring $R$ as following: let $\phi:L(R)\rightarrow
L(R)\cup\{\emptyset\}$ be a function, where $L(R)$ denotes the set of all
ideals of $R$. A proper ideal $I$ of a commutative ring $R$ is called $\phi
$-\textit{prime} if for $a_{1},a_{2}\in R$ with $a_{1}a_{2}\in I-\phi(I)$,
then $a_{1}\in I$ or $a_{2}\in I.$ They defined the map $\phi_{\alpha
}:L(R)\rightarrow L(R)\cup\{\emptyset\}$ as follows:%

\begin{align*}
\phi_{\emptyset}(I)  &  =\emptyset\ \ \ \ \ \ \ \ \ \ \ \ \text{prime ideal}\\
\phi_{0}(I)  &  =0_{R}\ \ \ \ \ \ \ \ \ \ \text{weakly prime ideal}\\
\phi_{2}(I)  &  =I^{2}\ \ \ \ \ \ \ \ \ \ \ \text{almost prime ideal}\\
\phi_{n}(I)  &  =I^{n}\ \ \ \ \ \ \ \ \ \ \ n\text{-almost prime ideal
	\ }(n\geq2)\\
\phi_{\omega}(I)  &  =\cap_{n=1}^{\infty}I^{n}\ \ \ \ \omega\text{-prime
	ideal}\\
\phi_{1}(I)  &  =I\ \ \ \ \ \ \ \ \ \ \ \ \ \text{any ideal.}%
\end{align*}

\bigskip

\bigskip\ The principal ideal generated by $a\in R$ is denoted by $<a>.$ For a
subset $S$ of $R$ and an ideal $I$ of $R,$ we define $(I:_{R}S):=\{r\in
R:rS\subseteq I\}.$ In particular, we use $Ann(S)$ instead of $(0:_{R}S).$
Also, for any $a\in R$ and any ideal $I$ of $R,$ we use $(I:a)$ and
$Ann(a)\ $to denote $(I:_{R}\{a\})$ and$\ Ann(\{a\}),$ respectively. An
element $a\in R$ is called a regular (resp., zerodivisor) element if
$Ann(a)=(0)$ $($resp., $Ann(a)\neq(0)).$ The set of all regular (resp.,
zerodivisor) elements of $R$ is denoted by $r(R)$ $($resp., $zd(R)).$ Let $S$
be a multiplicatively closed of $R.\ S^{-1}R$ denotes the quotient ring of
$R\ $at $S.\ $Particularly,$\ (r(R))^{-1}R=T(R)\ $is the total quotient ring
of $R.\ $A ring $R\ $is said to be a total quotient ring if $R=T(R),\ $%
equaivalently, every element $a\in R\ $is either zerodivisor or unit. An ideal
$I$ of $R$\ is called a regular ideal if it contains at least a regular
element. For a subset $S$ of $R,$ $a\in S$ is called a von Neumann regular
elemet if there is $b\in S$ such that $a=a^{2}b.$ If all elements of $R$ are
von Neumann regular element, $R$ is called a von Neumann regular ring.
Similarly, if all elements in a subset $S$ (resp., an ideal $I$) of $R$ are
von Neumann regular element, $S$ \ (resp., $I$) is called a von Neumann
regular subset (resp., ideal). An ideal $I$ is said to be pure ideal if for
every $a\in I$ there is $b\in I$ such that $a=ab.$ A ring $R\ $is said to
satisfy strongly annihilator condition (briefly, s. a. c.) if for every
finitely generated (briefly, f. g.) ideal $I$ of $R$ there is an element $b\in
I$ such that $Ann(I)=Ann(b).$ For more information about the above notions, we
refer to \cite{Gil}, \cite{Huc}.

\bigskip In 2015, R. Mohamadian \cite{rostam} present the notion of $r$-ideals
in commutative rings with nonzero identity as follows: an ideal $I$ is a
commtative ring with identity $R$ is called $r$-ideal (resp., $pr$-ideal), if
whenever $ab\in I$ and $a$ is regular element imply that $b\in I$ (resp.,
$b^{n}\in I,$ for some natural number $n),$ for each $a,b\in R.$

\bigskip

In this paper, our aim is to introduce the generalization of the concepts of
$r$-ideal, $pr$-ideal, pure ideal and von Neumann regular ideal in commutative
rings with nonzero identity. For this, firstly with Definition \ref{def 1} we
give the definition of $\phi-r$-ideal. Similarly, in Definition \ref{def 1_1}
we give the definitions of $\phi-pr$-ideal, $\phi-$pure ideal and $\phi-$von
Neumann regular ideal in $R.$ Then we investigate the basic properties of
$\phi-r$-ideal of $R,$ see Proposition \ref{pro basic} and Remark
\ref{remark}$.$ In Theorem \ref{the cha}, we give a method for constructing
$\phi-r$-ideal in a commutative ring with non-zero identity. Also, if
$\phi(I)$ is an $r$-ideal of $R,$ we prove that $I$ is a $\phi-r$-ideal of
$R\Leftrightarrow$ $I$ is $r$-ideal of $R,$ see Theorem \ref{the equ}. Let $I$
be a proper ideal of $R$ with $\sqrt{\phi(I)}=\phi(\sqrt{I}).$ If $I$ is a
$\phi-r$-ideal of $R$, then $\sqrt{I}$ is a $\phi-r$-ideal of $R.$ With
Theorem \ref{the rad}, under the condition $\sqrt{\phi(I)}=\phi(\sqrt{I}),$ we
obtain if $I$ is a $\phi-r$-ideal of $R$, then $\sqrt{I}$ is a $\phi-r$-ideal
of $R.$ In Theorem \ref{the union}, we show that if $\phi$ preserves the
order, then $\underset{k\in\Delta}{\bigcup}I_{k}$ is a $\phi-r$-ideal of $R$
where $\{I_{k}\}_{k\in\Delta}$ is a collection of $\phi-r$-ideal of $R.$
Moreover, in Theorem \ref{the ide}, we examine the concept of $\phi-r$-ideal
in $R(+)M,$ that is, the trivial ring extension, where $M\ $is an $R$-module.
Also, we examine the notion of $\phi-r$-ideal in $S^{-1}R,$ where $S$ is a
multiplicatively subset of $R$, see Theorem \ref{the loc}. Finally, we
characterize total quotient rings in terms of $\phi-r$-ideals.

\section{$\phi-r$-ideals of Commutative Rings}
\label{sec:1}

\begin{definition}
	\label{def 1}Let $R$ be a commutative ring with nonzeo identity and $I$ be a
	proper ideal of $R$. Let $\phi:L(R)\rightarrow L(R)\cup\{\emptyset\}$ be a
	function. If $ab\in I-\phi(I)$ and $Ann(a)=(0)$ imply that $b\in I$ for each
	$a,b\in R,$ then $I$ is called a $\phi-r$-ideal.
\end{definition}

\begin{definition}
	\label{def 2}If we define the map $\phi_{\alpha}:L(R)\rightarrow
	L(R)\cup\{\emptyset\}$ as the followings,
\end{definition}

\begin{align*}
\phi_{\emptyset}(I)  &  =\emptyset\ \ \ \ \ \ \ \ \ \ \ \phi_{\emptyset
}-r\text{-ideal (}r\text{-ideal)}\\
\phi_{0}(I)  &  =0_{R}\ \ \ \ \ \ \ \ \ \phi_{0}-r\text{-ideal (weakly
}r\text{-ideal)}\\
\phi_{1}(I)  &  =I\text{ \ \ \ \ \ \ \ \ \ \ }\phi_{1}-r\text{- ideal (any
	ideal)}\\
\phi_{2}(I)  &  =I^{2}\ \ \ \ \ \ \ \ \ \ \phi_{2}-r\text{-ideal (almost
}r\text{-ideal) }\\
\phi_{n}(I)  &  =I^{n}\ \ \ \ \ \ \ \ \ \ \phi_{n}-r\text{-ideal
	(}n\text{-almost }r\text{-ideal)\ }(n\geq2)\\
\phi_{\omega}(I)  &  =\cap_{n=1}^{\infty}I^{n}\ \ \ \phi_{\omega
}-r\text{-ideal (}\omega-r\text{-ideal) }%
\end{align*}

\bigskip Observe that $2$-almost $r$-ideals are exactly almost $r$-ideals.

\begin{definition}
	\label{def 1_1}Let $R$ be a commutative ring with nonzero identity and $I$ be
	a proper ideal of $R$. Let $\phi:L(R)\rightarrow L(R)\cup\{\emptyset\}$ be a function.
	
	\begin{enumerate}
		\item If $ab\in I-\phi(I)$ and $Ann(a)=(0)$ imply that $b^{n}\in I$ for some
		natural number $n$ and $a,b\in R,$ then $I$ is called a $\phi-pr$-ideal.
		
		\item If for every $a\in I-\phi(I)$ there is $b\in I$ such that $a=ab,$ then
		$I$ is called a $\phi-$pure ideal.
		
		\item If for every $a\in I-\phi(I)$ there is $b\in I$ such that $a=a^{2}b,$
		then $I$ is called a $\phi-$von Neumann regular ideal.
	\end{enumerate}
\end{definition}

\bigskip

Moreover, we define the concepts in Definition \ref{def 2} for $\phi
-pr$-ideal, $\phi-$pure ideal and $\phi-$von Neuman regular ideal.

\begin{remark}
	\label{remark}The elementary properties satisfy:
	
	\begin{enumerate}
		\item An ideal is weakly $r$-ideal if and only if it is an $r$-ideal.
		
		\item Every $\phi-r$-ideal is a $\phi-pr$-ideal.
		
		\item Every $\phi-$pure ideal is a $\phi-r$-ideal.
		
		\item Every $\phi-$von Neumann regular ideal is a $\phi-r$-ideal.
	\end{enumerate}
\end{remark}

\bigskip

Throughout this paper $\phi:L(R)\rightarrow L(R)\cup\{\emptyset\}$ is a
function. Since $I-\phi(I)=I-(I\cap\phi(I))$, for any ideal $I$ of $R$,
without loss of generality, assume that $\phi(I)\subseteq I.$ Moreover, let
$\psi_{1},$ $\psi_{2}:L(R)\rightarrow L(R)\cup\{\emptyset\}$ be two functions,
if $\psi_{1}(I)\subseteq\psi_{2}(I)$ for each $I\in L(R),$ we denote $\psi
_{1}\leq\psi_{2}.$ Thus clearly, we have the following order: $\phi
_{\emptyset}\leq\phi_{0}\leq\phi_{\omega}\leq\cdots\leq\phi_{n+1}\leq\phi
_{n}\leq\cdots\leq\phi_{2}\leq\phi_{1}$.

\begin{proposition}
	\label{pro basic}Let $R$\ be a ring and $I$ be a proper ideal $R.$ Let
	$\psi_{1},\psi_{2}:L(R)\rightarrow L(R)\cup\{\emptyset\}$ be two functions
	with $\psi_{1}\leq\psi_{2}.$
	
	\begin{enumerate}
		\item If $I$ is a $\psi_{1}$-$r$-ideal of $R,$ then $I$ is a $\psi_{2}$%
		-$r$-ideal of $R.$
		
		\item $I$ is a $r$-ideal $\Leftrightarrow$ $I$ is a weakly $r$-ideal
		$\Rightarrow$ $I$ is an $\omega$-$r$-ideal $\Rightarrow$ $I$ is an
		$(n+1)$-almost $r$-ideal$\Rightarrow$ $I$ is an $n$-$r$-ideal $(n\geq
		2)\Rightarrow$ $I$ is an almost $r$-ideal.
		
		\item $I$ is an $\omega$-$r$-ideal if and only if $I$ is an $n$-almost
		$r$-ideal for each $n\geq2.$
		
		\item $I$ is an idempotent ideal of $R$ $\Rightarrow$ $I$ is an $\phi_{n}-r
		$-ideal of $R$ for every $n\geq1.$
	\end{enumerate}
\end{proposition}

\begin{proof}
	(1): It is evident.
	
	(2): Follows from (1).
	
	(3): Every $\omega$-$r$-ideal is an $n$-almost $r$-ideal ideal for each
	$n\geq2\ $since $\phi_{\omega}\leq\phi_{n}.\ $Now, let $I\ $be an $n$-almost
	$r$-ideal for each $n\geq2.\ $Choose two elements $a,b\in R$ such that$\ ab\in
	I-$ $\cap_{n=1}^{\infty}I^{n}$ and $a$ is regular.$\ $Then we have $ab\in
	I-I^{n}\ $for some $n\geq2.\ $Since $I\ $is an $n$-almost $r$-ideal of
	$R\ $and $a$ is regular, we conclude $b\in I.\ $Therefore, $I\ $is an $\omega
	$-$r$-ideal.
	
	(4): Since $I^{n}=I$ for $n\geq2,$ it is clear.
\end{proof}

\begin{theorem}
	Let $I$ be a proper ideal $R.$
	
	\begin{enumerate}
		\item Let $\phi(I)\subseteq r(R).$ If $I$ is a $\phi-r$-ideal of $R,$ then
		$I/\phi(I)$ is an $r$-ideal of $R/\phi(I).$
		
		\item Let $\phi(I)$ be an $r$-ideal of $R$. If $I/\phi(I)$ is an $r$-ideal of
		$R/\phi(I),$ then $I$ is a $\phi-r$-ideal of $R$
	\end{enumerate}
\end{theorem}

\begin{proof}
	$(1):$ Suppose that $I$ is a $\phi-r$-ideal of $R.$ Let $a+\phi(I),b+\phi
	(I)\in R/\phi(I)$ such that $(a+\phi(I))(b+\phi(I))\in I/\phi(I)$ and
	$a+\phi(I)$ be a regular element of $R/\phi(I)$. Therefore $ab\in I-$
	$\phi(I).$ Since $a+\phi(I)$ is regular element, $a$ is a regular element of
	$R$. Indeed, if $a$ is non-regular, there exists $0\neq x\in R$ such that
	$xa=0. $ Then we have $(a+\phi(I))(x+\phi(I))=0+\phi(I).$ Since $a+\phi(I)$ is
	regular, $x+\phi(I)=0+\phi(I),$ i.e., $x\in\phi(I)\subseteq r(R).$ This gives
	us $Ann(x)=(0),$ a contradiction with $0\neq a\in Ann(x)$. Thus as $I$ is
	since $\phi-r$-ideal of $R,$ we obtain $b\in I.$ This implies that
	$b+\phi(I)\in I/\phi(I).$ It is done.
	
	$(2):$ Let $I/\phi(I)$ be an $r$-ideal of $R/\phi(I).$ Choose $a,b\in R$ such
	that $ab\in I-$ $\phi(I)$ which $a$ is regular. Then $(a+\phi(I))(b+\phi
	(I))\in I/\phi(I).$ Also, since $a$ is regular, $a+\phi(I)$ is regular.
	Indeed, if $a+\phi(I)$ is non-regular, there exists $0\neq x+\phi(I)\in$
	$R/\phi(I)$ such that $(a+\phi(I))(x+\phi(I))=0+\phi(I).$ Then $ax\in\phi(I).$
	Since $\phi(I)$ is $r$-ideal and $a$ is regular, we get $x\in\phi(I),$ a
	contradiction. Therefore, since $a+\phi(I)$ is regular and $I/\phi(I)$ is an
	$r$-ideal, we have $b+\phi(I)\in I/\phi(I),$ so $b\in I.$
\end{proof}

\begin{theorem}
	\label{the cha}Let $I$ be a proper ideal $R.$ Then the followings are equivalent:
	
	\begin{enumerate}
		\item $I$ is a $\phi-r$-ideal of $R.$
		
		\item For every regular element $x$ of $R$, $(I:_{R}x)=I\cup$ $(\phi
		(I):_{R}x).$
		
		\item For every regular element $x$ of $R$, $(I:_{R}x)=I$ or $(I:_{R}%
		x)=(\phi(I):_{R}x).$
	\end{enumerate}
\end{theorem}

\begin{proof}
	$(1)\Rightarrow(2):$ It is clear that $I\subseteq(I:_{R}x)$ and $(\phi
	(I):_{R}x)\subseteq(I:_{R}x).$ So, the first containment is obtained. For the
	other, choose $y\in(I:_{R}x).$ Then $xy\in I.$ If $xy\in\phi(I),$ $y\in
	(\phi(I):_{R}x).$ If $xy\notin\phi(I),$ as $I$ is $\phi-r$-ideal and $x$ is
	regular, $y\in I.$ Consequently, $y\in I\cup$ $(\phi(I):_{R}x).$
	
	$(2)\Rightarrow(3):$ It is obvious.
	
	$(3)\Rightarrow(1):$ Choose $y,z\in R$ such that $yz\in I-$ $\phi(I)$ which
	$y$ is a regular element. By (3), we conclude that $(I:_{R}y)=I$ or
	$(I:_{R}y)=(\phi(I):_{R}y).$ Let $(I:_{R}y)=I$. Then as $yz\in I,$ we get
	$z\in$ $(I:_{R}y)=I,$ as needed. Let $(I:_{R}y)=(\phi(I):_{R}y).$ Since $z\in$
	$(I:_{R}y),$ we say $yz\in\phi(I).$ This gives us a contradiction.
\end{proof}

\begin{proposition}
	Let $I$ be a proper ideal $R.$ If $I$ is a $\phi-r$-ideal of $R,$ then $I-$
	$\phi(I)\subseteq zd(R).$
\end{proposition}

\begin{proof}
	Suppose that $I-$ $\phi(I)\nsubseteq zd(R).$ Then there exists a regular
	element which $x\in I-$ $\phi(I).$ Consider $x$ as $x\cdot1_{R}.$ Then as $I$
	is $\phi-r$-ideal, we get $1_{R}\in I.$ This contradicts with the choice of
	$I$. As conclusion, $I-$ $\phi(I)\subseteq zd(R).$
\end{proof}

\begin{theorem}
	\label{the equ}Let $\phi(I)$ be an $r$-ideal of $R.$ Then $I$ is a $\phi
	-r$-ideal of $R$ $\Leftrightarrow$ $I$ is $r$-ideal of $R.$
\end{theorem}

\begin{proof}
	$\Rightarrow:$ Assume that $a,b\in R$ such that $ab\in I$ and$\ a$ is a
	regular element. If $ab\notin\phi(I),$ since $I$ is $\phi-r$-ideal, we have
	$b\in I,$ as desired. If $ab\in\phi(I),$ as $\phi(I)$ is $r$-ideal, we
	conclude that $b\in\phi(I)\subseteq I.$ Thus it is done.
	
	$\Leftarrow:$ Obvious.
\end{proof}

\begin{theorem}
	\label{the rad}Let $I$ be a proper ideal of $R$ with $\sqrt{\phi(I)}%
	=\phi(\sqrt{I}).$ If $I$ is a $\phi-r$-ideal of $R$, then $\sqrt{I}$ is a
	$\phi-r $-ideal of $R.$
\end{theorem}

\begin{proof}
	Choose $a,b\in R$ such that $ab\in\sqrt{I}-$ $\phi(\sqrt{I})$ and $a$ is
	regular. As $ab\in\sqrt{I},$ there exists a natural number $k$ which
	$(ab)^{k}=a^{k}b^{k}\in I.$ On the other hand, $ab\notin$ $\phi(\sqrt
	{I})=\sqrt{\phi(I)},$ one can say $a^{k}b^{k}\notin$ $\phi(I).$ Also, since
	$a$ is regular, it is obvious that $a^{k}$ is regular. As $I$ is $\phi
	-r$-ideal, we get $b^{k}\in I,$ i.e., $b\in\sqrt{I}.$
\end{proof}

\begin{theorem}
	\label{the union}Let $\{I_{k}\}_{k\in\Delta}$ be a collection of ascending
	chain of $\phi-r$-ideals of $R.$ If $\phi$ preserves the order, then
	$\underset{k\in\Delta}{\bigcup}I_{k}$ is a $\phi-r$-ideal of $R.$
\end{theorem}

\begin{proof}
	Choose $x,y\in R$ such that $xy\in\underset{k\in\Delta}{\bigcup}I_{k}-$
	$\phi(\underset{k\in\Delta}{\bigcup}I_{k})$ which $x$ is a regular element.
	Then for some $i\in\Delta,$ $xy\in I_{i}.$ It is clear that $xy\notin
	\phi(I_{i}),$ since $\phi$ preserves the order. Thus, as $I_{i}$ is a $\phi
	-r$-ideal of $R,$ we see $y\in I_{i}\subseteq\underset{k\in\Delta}{\bigcup
	}I_{k}.$
\end{proof}

\begin{proposition}
	\label{pro zd}Let $I$ be a proper ideal of $R$ and $\phi(I)$ be an $r$-ideal
	of $R.$ If $I$ is a $\phi-r$-ideal of $R,$ then $I\subseteq zd(R).$
\end{proposition}

\begin{proof}
	Suppose that $I\nsubseteq zd(R).$ There is an element $x\in I$ but $x\notin
	zd(R).$ So $x$ is a regular element. Consider $x\cdot1_{R}\in I.$ If
	$x\notin\phi(I),$then $1_{R}\in I,$ a contradiction. If $x\in$ $\phi(I),$ we
	obtain $1_{R}\in\phi(I)\subseteq I,$ since $\phi(I)$ is an $r$-ideal of $R.$
\end{proof}

\begin{proposition}
	Let $\phi(I)$ be an $r$-ideal of $R$ and $I$ be a prime ideal of $R.$ Then $I$
	is $\phi-r$-ideal $\Leftrightarrow$ $I\subseteq zd(R).$
\end{proposition}

\begin{proof}
	Suppose $I$ is a prime ideal and $\phi(I)$ is an $r$-ideal of $R$.
	
	$\Rightarrow:$ It is clear by Proposition \ref{pro zd}.
	
	$\Leftarrow:$ Choose $a,b\in R$ such that $ab\in I-\phi(I)$ and $a$ is a
	regular elemet$.$ Since $I$ is prime, either $a\in I$ or $b\in I.$ The first
	option contradicts with $a\in r(R),$ since $I\subseteq zd(R).$ Consequently,
	$b\in I,$ as desired.
\end{proof}

\begin{proposition}
	\label{pro X}Let $I$ be a proper ideal of $R$ and $x\in R-I.$ Let
	$(\phi(I):_{R}x)\subseteq\phi((I:_{R}x)).$ If $I$ is a $\phi-r$-ideal of $R$,
	then\ $(I:_{R}x)$ is $\phi-r$-ideal of $R$.
\end{proposition}

\begin{proof}
	Set $y,z\in R$ such that $yz\in(I:_{R}x)-\phi((I:_{R}x))$ and $y$ is regular.
	Then we get $yzx\in I$ and $yzx\notin\phi(I)$ by $(\phi(I):_{R}x)\subseteq
	\phi((I:_{R}x)).$ Since $y$ is regular and $I$ is a $\phi-r$-ideal, we get
	$zx\in I,$ i.e., $z\in(I:_{R}x),$ as required.
\end{proof}

\begin{definition}
	Let $I$ be a proper ideal of $R$. If for every ideal $X$ and $Y$ of $R$ such
	that $XY\subseteq I,$ $XY\nsubseteq\phi(I)$ and $Ann(X)=(0)$ implies
	$Y\subseteq I,$ then $I$ is called strongly $\phi-r$-ideal of $R.$
\end{definition}

\begin{proposition}
	Every strongly $\phi-r$-ideal is a $\phi-r$-ideal.
\end{proposition}

\begin{proof}
	Suppose that $I$ is a strongly $\phi-r$-ideal. Choose $x,y\in R$ such that
	$xy\in I-\phi(I)$ and $x$ is regular. Then it is clear that $<x><y>\subseteq
	I$ and $<x><y>\nsubseteq\phi(I).$ Now let us observe that $Ann(<x>)=(0).$
	Suppose that for $0\neq u\in R,$ $u<x>=(0).$ This means that for all $r\in R,
	$ $urx=0.$ Since$\ x$ is regular, we obtain $ur=0$ for all $r\in R.$ This
	implies $u=0,$ a contradiction. Thus as $I$ is strongly $\phi-r$-ideal,
	$<y>\subseteq I$ , so $y\in I.$
\end{proof}

\begin{theorem}
	Let $R$ satisfy the s. a. c. and $I$ be a proper ideal. Let $\phi(I)$ be an
	$r$-ideal of $R.$ $I$ is a $\phi-r$-ideal if and only if for every f. g. ideal
	$X$ of $R$ and every ideal $Y$ of $R$ such that $XY\subseteq I,$
	$XY\nsubseteq\phi(I)$ and $Ann(X)=(0)$ implies $Y\subseteq I.$
\end{theorem}

\begin{proof}
	($\Rightarrow$):\ Assume that $I$ is a $\phi-r$-ideal and $\phi(I)$ is a
	$r$-ideal. Let $X$ \ be a f. g. ideal of $R$ and $Y$ be an ideal of $R$ such
	that $XY\subseteq I,$ $XY\nsubseteq\phi(I)$ and $Ann(X)=(0).$ Suppose
	$Y\nsubseteq I.$ Then there is an element $0\neq y\in Y$ and $y\notin I.$ On
	the other hand, since $R$ satisfies the s. a. c. and $\ X$ is f.g., there is
	$z\in X$ with $Ann(X)=Ann(z).$ This means that $z$ is a regular element, as
	$Ann(X)=(0).$ Consider $zy\in I.$ If $zy\notin\phi(I),$ since $I$ is a
	$\phi-r$-ideal, we obtain $y\in I,$ a contradiction. If $zy\in\phi(I),$ as
	$\phi(I)$ is a $r$-ideal, again we conclude the same contradiction. Thus it
	must be $Y\subseteq I.\ $
	
	($\Leftarrow$):\ It is straightforward.
\end{proof}

\bigskip In \cite{rostam} by the help of Proposition 2.22, R. Mohammadian
proved that if $J/I$ is an $r$-ideal of $R/I,$ then $J$ is also an $r$-ideal
of $R,$ where $I$ is an $r$-ideal of $R$ contained in ideal $J.$ Now, let us
examine the proposition for the concept of $\phi-r$-ideal.

\bigskip Let $I$ be an ideal of $R.$ Define $\phi_{I}:L(R/I)\rightarrow
L(R/I)\cup\{\emptyset\}$ by $\phi_{I}(J/I)=(\phi(J)+I)/I$ for every ideal
$I\subseteq J$ and $\phi_{I}(J/I)=\emptyset$ if $\phi(J)=\emptyset. $ Notice
that $\phi_{I}(J/I)\subseteq J/I.$

\begin{theorem}
	\bigskip Let $I\subseteq r(R)$ be an ideal of $R$ contained in ideal $J.$ If
	$J$ is a $\phi-r$-ideal of $R$, then $J/I$ is a $\phi_{I}-r$-ideal of $R/I. $
\end{theorem}

\begin{proof}
	Let $a+I,b+I\in R/I$ such that $(a+I)(b+I)\in J/I-\phi_{I}(J/I)$ and $a+I$ is
	a regular element. Then we see $ab\in J-\phi(J).$ Also, since $a+I$ is a
	regular element, $a$ is a regular element. Indeed, if $a$ is non-regular,
	there is $0\neq x\in R$ such that $xa=0.$ Then we have $(a+I)(x+I)=0+I.$ Since
	$a+I$ is regular, $x+I=0+I,$ i.e., $x\in I\subseteq r(R).$ This gives us a
	contradiction as $0\neq a\in Ann(x)$. Thus, as $J$ is a $\phi-r$-ideal, $b\in
	J,$ so $b+I\in J/I.$
\end{proof}

\begin{proposition}
	Let $I$ be a $r$-ideal of $R$ contained in ideal $J.$ If $J/I$ is an $r$-ideal
	of $R/I,$ then $J$ is a $\phi-r$-ideal of $R.$
\end{proposition}

\begin{proof}
	Choose $x,y\in R$ such that $xy\in J-$ $\phi(J)$ which $x$ is a regular
	element. Then we have 2 cases:
	
	Case 1: Let $xy\in I.$ Then as $I$ is a $r$-ideal, we get $y\in I\subseteq J,
	$ as desired.
	
	Case 2: Let $xy\in J-I.$ This implies that $xy+I=(x+I)(y+I)\in J/I$ . Also, as
	$x$ is a regular element, $x+I$ is regular. Indeed, if $x+I$ is non-regular,
	there is an element $0\neq a+I\in R/I$ such that $(x+I)(a+I)=0+I.$ This means
	that $xa\in I.$ But as $x$ regular and $I$ is $r$-ideal, we see $a\in I,$ a
	contradiction. Therefore, since $J/I$ is an $r$-ideal of $R/I$ and $x+I$ is
	regular, we see $y+I\in J/I,$ so $y\in J.$\bigskip
\end{proof}

Let $R$ be a commutative ring with nonzero identity and $M$ be an $R$-module.
Then the idealization, $R(+)M=\{(a,m):a\in R,m\in M\}$ is a commutative ring
with componentwise addition and the multiplication $(a,m)(b,n)=(ab,an+bm)$ for
each $a,b\in R$ and $m,n\in M.$ In addition, if $J\ $is an ideal of $R\ $and
$N\ $is a submodule of $M,\ $then $J(+)N\ $is an ideal of $R(+)M\ $if and only
if $JM\subseteq N\ $\cite{AnWi}. Also note that $z(R(+)M)=\{(r,m):r\in
z(R)\cup z(M)\},\ $where $z(M)=\{r\in R:rm=0\ $for some $0\neq m\}\ $%
\cite{AnWi}.

\begin{theorem}
	\label{the ide}Let $M$ be an $R$-module such that $z(R)=z(M)$. Let $\psi
	_{1}:L(R)\rightarrow L(R)\cup\{\emptyset\}$ and $\psi_{2}:L(R(+)M)\rightarrow
	L(R(+)M)\cup\{\emptyset\}$ be two function such that $\psi_{2}(I(+)M)=\psi
	_{1}(I)(+)M$ for a proper ideal $I$ of $R.$ If $I(+)M$ is a $\psi_{2}-r$-ideal
	of $R(+)M,$ then$\ I$ is a $\psi_{1}-r$-ideal of $R$.
\end{theorem}

\begin{proof}
	Let $a,b\in R$ such that $ab\in I-$ $\psi_{1}(I)$ and $a$ is regular. Since
	$\psi_{2}(I(+)M)=\psi_{1}(I)(+)M,$ we have $(a,0)(b,0)\in I(+)M-\psi
	_{2}(I(+)M).$ Also, since $z(R)=z(M),$ it is clear that $(a,0)$ is a regular
	element of $R(+)M.$ This means that $(b,0)\in I(+)M,$ so $b\in I,$ as needed.
\end{proof}

\bigskip Let $S$ be a multiplicatively closed subset of $R$. Then
$S^{-1}R:=\{\frac{r}{s}:r\in R,s\in S\}$ is called the localization of $R$ at
the multiplicatively closed set $S$ and it is denoted by $R_{S}.$ Let
$\phi:L(R)\rightarrow L(R)\cup\{\emptyset\}$ be a function. Define $\phi
_{S}:L(R_{S})\rightarrow L(R_{S})\cup\{\emptyset\}$ by $\phi_{S}(J)=S^{-1}%
\phi(J\cap R)$ for every ideal $J$ of $R_{S}$ and $\phi_{S}(J)=\emptyset$ if
$\phi(J\cap R)=\emptyset.$ Notice that $\phi_{S}(J)\subseteq J.$

\begin{theorem}
	\label{the loc}Let $\phi:L(R)\rightarrow L(R)\cup\{\emptyset\}$ be a function
	and $S$ be a multiplicatively closed subset of $R$ such that $S\subseteq
	r(R).\ $Suppose that $S\cap I=\emptyset$ for an ideal $I$ of $R.$ If $I$ is a
	$\phi-r$-ideal and $\phi(I)_{S}\subseteq\phi_{S}(I_{S}),$ then the followings
	are hold:
	
	\begin{enumerate}
		\item $I_{S}$ is a $\phi_{S}-r$-ideal of $R_{S}.$
		
		\item If $I_{S}\neq\phi(I)_{S},$ then $I_{S}\cap R\subseteq zd(R).$
	\end{enumerate}
\end{theorem}

\begin{proof}
	(1): Choose $a/s,b/t\in R_{S}$ such that $\frac{a}{s}\frac{b}{t}\in I_{S}%
	-\phi_{S}(I_{S})$ and $\frac{a}{s}$ is regular. Then there is $s_{1}\in S$
	such that $s_{1}ab\in I.$ Also, it is clear that since $\frac{a}{s}$ is
	regular, $a$ is regular. On the other hand, $\frac{a}{s}\frac{b}{t}\notin
	\phi_{S}(I_{S})$ implies that $s_{2}ab\notin\phi_{S}(I_{S})\cap R$\ for all
	$s_{2}\in S.$ By $\phi(I)_{S}\subseteq\phi_{S}(I_{S})$, we have $s_{2}%
	ab\notin\phi(I)$ for all $s_{2}\in S.$ Thus we can say $as_{1}b\in I-\phi(I).$
	Then as $I$ is $\phi-r$-ideal, $s_{1}b\in I.$ Hence $\frac{b}{t}=\frac{s_{1}%
		b}{s_{1}t}\in I_{S},$ as required.
	
	(2): Set $a\in I_{S}\cap R.$ Then there is $s\in S$ such that $as\in I.$ Also,
	$s\notin I$ by $S\cap I=\emptyset$. If $as\notin\phi(I),$ then we obtain$\ a$
	is non-regular since $s\notin I.$ If $as\in\phi(I),$ then $a\in\phi(I)_{S}\cap
	R.$ Thus $I_{S}\cap R\subseteq zd(R)\cup(\phi(I)_{S}\cap R).$ Then by our
	assumption $I_{S}\neq\phi(I)_{S},$ we obtain $I_{S}\cap R\subseteq zd(R).$
\end{proof}

Let $\phi_{i}:L(R_{i})\rightarrow L(R_{i})\cup\{\emptyset\}$ be a function for
each $i=1,2,\ldots,m$ and $R=R_{1}\times R_{2}\times\cdots\times R_{m}%
,\ $where $R_{1},R_{2},\ldots,R_{m}\ $are commutative rings$.\ $Then
$\phi_{\times}:L(R)\rightarrow L(R)\cup\{\emptyset\}$, defined by
$\phi^{\times}(I_{1}\times I_{2}\times\cdots\times I_{m})=\phi_{1}%
(I_{1})\times\phi_{2}(I_{2})\times\cdots\times\phi_{m}(I_{m}),\ $becomes a
function.\ Particularly, if each $\phi_{i}$ is the function $\phi_{n},\ $then
$\phi^{\times}-r$-ideal is denoted by $\phi_{n}^{\times}-r$-ideal.

\begin{theorem}
	Let $m\geq1\ $and $n\geq2\ $be two integers. Suppose that $R=R_{1}\times
	R_{2}\times\cdots\times R_{m},\ $where $R_{1},R_{2},\ldots,R_{m}\ $are
	commutative rings.\ The following statements are equivalent:
	
	(i) $R_{1},R_{2},\ldots,R_{m}\ $are total quotient rings.
	
	(ii)\ Every proper ideal of $R\ $is an $\phi_{n}^{\times}-r$-ideal.
\end{theorem}

\begin{proof}
	$(i)\Rightarrow(ii):\ $Suppose that $R_{1},R_{2},\ldots,R_{m}\ $are total
	quotient rings so that every element of $R\ $is either zerodivisor or unit.
	Thus every ideal ideal of $R\ $is a trivially $\phi_{n}^{\times}-r$-ideal.
	
	$(ii)\Rightarrow(i):\ $Assume that every ideal of $R\ $is an $\phi_{n}%
	^{\times}-r$-ideal. Let $a\in R_{1}\ $be a regular element. Now, we will show
	that $a\ $is unit. Suppose to the contrary. Put $I=(a^{2})\times R_{2}%
	\times\cdots\times R_{n}.\ $Then $I\ $is proper so is a $\phi_{n}^{\times}%
	-r$-ideal by assumption. Now, put $x=(a,1,1,\ldots,1)$ and $y=(a,1,\ldots
	,1).\ $Then $xy=(a^{2},1,1,\ldots,1)\in I\ $and $x$ is regular. If $xy\in
	\phi_{n}^{\times}(I),\ $we conclude that $(a^{2},1,1,\ldots,1)\in
	(a^{2n})\times R_{2}\times\cdots\times R_{n}$ so that $a^{2}=a^{2n}t$ for some
	$t\in R.\ $As $a$ is regular, we conclude that $1=a(a^{2n-3}t)$ and thus $a$
	is unit. So assume that $xy\in I-\phi_{n}^{\times}(I).\ $Since $I\ $is
	$\phi_{n}^{\times}-r$-ideal and $x$ is regular, we conclude that $y\in I\ $and
	so $a=a^{2}x$ for some $x\in R.\ $Since $a$ is regular, we have $1=ax$, that
	is, $a$ is unit. Thus $R_{1}\ $is a total quotient ring. One can similarly
	show that $R_{k}\ $is total quotient ring for each $2\leq k\leq m.$
\end{proof}


\end{document}